\documentclass[12pt, 14paper,reqno]{amsart}
\usepackage{amsmath,amsfonts,amssymb}
\usepackage{longtable}

\usepackage[mathscr]{eucal}

\usepackage{amssymb, amsmath, amsthm}
\usepackage[breaklinks]{hyperref}

\usepackage{graphicx}
\usepackage[mathscr]{eucal}
\usepackage{mathrsfs}
\usepackage{amsbsy}
\usepackage{dsfont}
\usepackage{bbm}
\usepackage{wasysym}
\usepackage{stmaryrd}
\usepackage{url}

\input xypic
\xyoption{all}

\makeindex
\makeglossary

\begin{document}
\baselineskip = 16pt

\newcommand \ZZ {{\mathbb Z}}
\newcommand \NN {{\mathbb N}}
\newcommand \RR {{\mathbb R}}
\newcommand \PR {{\mathbb P}}
\newcommand \AF {{\mathbb A}}
\newcommand \GG {{\mathbb G}}
\newcommand \QQ {{\mathbb Q}}
\newcommand \CC {{\mathbb C}}
\newcommand \bcA {{\mathscr A}}
\newcommand \bcC {{\mathscr C}}
\newcommand \bcD {{\mathscr D}}
\newcommand \bcF {{\mathscr F}}
\newcommand \bcG {{\mathscr G}}
\newcommand \bcH {{\mathscr H}}
\newcommand \bcM {{\mathscr M}}
\newcommand \bcJ {{\mathscr J}}
\newcommand \bcL {{\mathscr L}}
\newcommand \bcO {{\mathscr O}}
\newcommand \bcP {{\mathscr P}}
\newcommand \bcQ {{\mathscr Q}}
\newcommand \bcR {{\mathscr R}}
\newcommand \bcS {{\mathscr S}}
\newcommand \bcV {{\mathscr V}}
\newcommand \bcW {{\mathscr W}}
\newcommand \bcX {{\mathscr X}}
\newcommand \bcY {{\mathscr Y}}
\newcommand \bcZ {{\mathscr Z}}
\newcommand \goa {{\mathfrak a}}
\newcommand \gob {{\mathfrak b}}
\newcommand \goc {{\mathfrak c}}
\newcommand \gom {{\mathfrak m}}
\newcommand \gon {{\mathfrak n}}
\newcommand \gop {{\mathfrak p}}
\newcommand \goq {{\mathfrak q}}
\newcommand \goQ {{\mathfrak Q}}
\newcommand \goP {{\mathfrak P}}
\newcommand \goM {{\mathfrak M}}
\newcommand \goN {{\mathfrak N}}
\newcommand \uno {{\mathbbm 1}}
\newcommand \Le {{\mathbbm L}}
\newcommand \Spec {{\rm {Spec}}}
\newcommand \Gr {{\rm {Gr}}}
\newcommand \Pic {{\rm {Pic}}}
\newcommand \Jac {{{J}}}
\newcommand \Alb {{\rm {Alb}}}
\newcommand \Corr {{Corr}}
\newcommand \Chow {{\mathscr C}}
\newcommand \Sym {{\rm {Sym}}}
\newcommand \Prym {{\rm {Prym}}}
\newcommand \cha {{\rm {char}}}
\newcommand \eff {{\rm {eff}}}
\newcommand \tr {{\rm {tr}}}
\newcommand \Tr {{\rm {Tr}}}
\newcommand \pr {{\rm {pr}}}
\newcommand \ev {{\it {ev}}}
\newcommand \cl {{\rm {cl}}}
\newcommand \interior {{\rm {Int}}}
\newcommand \sep {{\rm {sep}}}
\newcommand \td {{\rm {tdeg}}}
\newcommand \alg {{\rm {alg}}}
\newcommand \im {{\rm im}}
\newcommand \gr {{\rm {gr}}}
\newcommand \op {{\rm op}}
\newcommand \Hom {{\rm Hom}}
\newcommand \Hilb {{\rm Hilb}}
\newcommand \Sch {{\mathscr S\! }{\it ch}}
\newcommand \cHilb {{\mathscr H\! }{\it ilb}}
\newcommand \cHom {{\mathscr H\! }{\it om}}
\newcommand \colim {{{\rm colim}\, }} 
\newcommand \End {{\rm {End}}}
\newcommand \coker {{\rm {coker}}}
\newcommand \id {{\rm {id}}}
\newcommand \van {{\rm {van}}}
\newcommand \spc {{\rm {sp}}}
\newcommand \Ob {{\rm Ob}}
\newcommand \Aut {{\rm Aut}}
\newcommand \cor {{\rm {cor}}}
\newcommand \Cor {{\it {Corr}}}
\newcommand \res {{\rm {res}}}
\newcommand \red {{\rm{red}}}
\newcommand \Gal {{\rm {Gal}}}
\newcommand \PGL {{\rm {PGL}}}
\newcommand \Bl {{\rm {Bl}}}
\newcommand \Sing {{\rm {Sing}}}
\newcommand \spn {{\rm {span}}}
\newcommand \Nm {{\rm {Nm}}}
\newcommand \inv {{\rm {inv}}}
\newcommand \codim {{\rm {codim}}}
\newcommand \Div{{\rm{Div}}}
\newcommand \CH{{\rm{CH}}}
\newcommand \sg {{\Sigma }}
\newcommand \DM {{\sf DM}}
\newcommand \Gm {{{\mathbb G}_{\rm m}}}
\newcommand \tame {\rm {tame }}
\newcommand \znak {{\natural }}
\newcommand \lra {\longrightarrow}
\newcommand \hra {\hookrightarrow}
\newcommand \rra {\rightrightarrows}
\newcommand \ord {{\rm {ord}}}
\newcommand \Rat {{\mathscr Rat}}
\newcommand \rd {{\rm {red}}}
\newcommand \bSpec {{\bf {Spec}}}
\newcommand \Proj {{\rm {Proj}}}
\newcommand \pdiv {{\rm {div}}}
\newcommand \wt {\widetilde }
\newcommand \ac {\acute }
\newcommand \ch {\check }
\newcommand \ol {\overline }
\newcommand \Th {\Theta}
\newcommand \cAb {{\mathscr A\! }{\it b}}

\newenvironment{pf}{\par\noindent{\em Proof}.}{\hfill\framebox(6,6)
\par\medskip}

\newtheorem{theorem}[subsection]{Theorem}
\newtheorem{conjecture}[subsection]{Conjecture}
\newtheorem{proposition}[subsection]{Proposition}
\newtheorem{lemma}[subsection]{Lemma}
\newtheorem{remark}[subsection]{Remark}
\newtheorem{remarks}[subsection]{Remarks}
\newtheorem{definition}[subsection]{Definition}
\newtheorem{corollary}[subsection]{Corollary}
\newtheorem{example}[subsection]{Example}
\newtheorem{examples}[subsection]{examples}

\title{Rank of elliptic curves and class groups of real quadratic fields}
\author{Kalyan Banerjee}
\email{kalyan.ba@srmap.edu.in}

\maketitle

\begin{abstract}
In this paper, we are going to prove the relation between rank of elliptic curves and the non-triviality of class groups of infinitely many real quadratic fields.
\end{abstract}

\section{Introduction}
The study of elliptic curves is an interesting area of research in number theory and arithmetic geometry. Elliptic curves are smooth plane curves given by the equation: 
$$y^2=x^3+ax+b$$
where $a,b$ are rational numbers with the property that
$$4a^3+27b^2\neq 0\;.$$
This ensures that we can draw a tangent at each and every point of the elliptic curve $E$ given by the above equation. Also, it should be noted that since it is a non-singular cubic curve, the curve is endowed with a group structure where there is a point at infinity serving as the identity of the group. 

On the other hand, there is another interesting object to study in the context of number theory, namely the ideal class group. This is the free abelian group generated by all the prime ideals in the ring of integers of a number field modulo the ideals which are principal. The folklore conjecture  says that there exist class groups of number fields with arbitrarily large order. The question is : is it somehow related to the torsions and rank (points of infinite order) of an elliptic curve? That is, given an elliptic curve, can we somehow construct a number field whose class group has an element of order $p$ for a large prime $p$.

In this paper, we address this question and the main theorem is as follows:

\begin{theorem}\label{thm1}Let $E$ be an elliptic curve given by equation $$y^2=x^3+ax+b$$ where $a,b$ are rational integers and $4a^3+27b^2\neq 0$. Let the Mordell-Weil group $E(\QQ)$ has positive rank and there exists a rational point on $E$, except the rank point. Then for an infinite number of prime numbers $p$, $p^3+ap+b$ is a square free positive integer and the class group of $\QQ(\sqrt{p^3+ap+b})$ contains a non-trivial element which is obtained by specializing a generator of the rank part of $E(\QQ)$. 
\end{theorem}

Here, the existence of infinitely many primes such that $p^3+ap+b$ is a non-square exists by the Siegel's theorem that the number of integral points on an elliptic curve with integer coefficients is finite, so the number of solutions of the equation $y^2=p^3+ap+b$ is finite and hence there are infinitely many primes $p$ such that $p^3+ap+b$ is non-square.

This problem has previously been extensively studied from an algebraic geometry perspective by \cite{GL}, \cite{GL-18}, \cite{AP}, \cite{So}.
In their approach, they had used the notion on Picard group of Cartier divisors to deduce the result about the connection between line bundles and elements of class groups. In addition, the approach of \cite{GI} is important for the hyperelliptic curves in the same context. Here the novelty is that we are using the Weil divisors in strak difference with Cartier divisors and this approach does extend to singular varieties and to produce class groups of large order from them as well. The other advantage that we have from the approach of Soleng \cite{So}, where a similar study has been done in a concrete way, is that the homomorphism defined from the elliptic curve to the class group here is  motivic and functorial in nature and can be generalized to higher dimensional varieties.  

\section{Proof of the theorem \ref{thm1}}
For the proof of this theorem we consider the technique of Chow schemes or Hilbert schemes on a scheme or variety defined over $\Spec(\ZZ)$. Let $E_{\ZZ}$ be a smooth integral model of the curve $E$ defined over $\Spec(\ZZ)$. Taking a point on the elliptic curve as a Weil divisor on the curve, that is, using an isomorphism $E\cong \Pic^0(E)$ we have such a realization, we spread the Weil divisor on $E_{\ZZ}$ to obtain a divisor on the fixed integral model of $E$ and then specialize this divisor at a prime integer $p$ of $Spec(\ZZ)$ to obtain an element of the ideal class group of a number field given by $\QQ(\sqrt{p^3+ap+b})$. The only thing is that if we start with an element of infinite order of the elliptic curve defined over $\QQ$ we have to prove that the ideal class element that it produces in $\QQ(\sqrt{p^3+ap+b})$ is non-principal.

Let us give some details on the spreading technique for Weil divisors on $E$. Let us take a point $(x,y)$ in $E(\QQ)$ which is a rank point. Since $y^2=x^3+ax+b$, we have the following. Write $x=k/l, y=m/n$, and then substituting into the equation we have
$$(\frac{m}{n})^2=(\frac{k}{l})^3+a\frac{k}{l}+b$$
Then by clearing the denominator we have 
$$l^3m^2=n^2k^3+an^2kl^2+bl^3n^2$$
The above equation is defined over the ring of integers and defines an arithmetic variety. Given the point $(x,y)$, it defines a point in this arithmetic variety and vise versa. Consider $Spec(\bcO_K)$ where $K=\QQ(\sqrt{p^3+ap+b})$, $p\geq 3$ is a prime number. There is a morphism of schemes from $\Spec(\bcO_K)$ to $\Spec \ZZ$, induced by the homomorphism $\ZZ\to \bcO_K$. Now consider the arithmetic variety and consider its specialization at $x=p$,  then we have the arithmetic variety 
$$m^2=n^2p^3+an^2p+bn^2$$
and its normal closure , which is exactly $\bcO_K$. Now given a rank point in $E(\QQ)$ we can think of it as a Weil divisor on $E$ of infinite order, using the spread and the specialization we get a divisor on $\bcO_K$ and hence an element in $cl(K)$.
 
First, notice that all elements of infinite  order  on $E$ defined over $\QQ$ are spread out to give a collection of elements on $E_{\ZZ}$. So, if we consider $C^1_{d,d}(E_{\ZZ})$ the Chow scheme of the integral model of the elliptic curve parametrizing codimension one cycles in $E_{\ZZ}$ such that it has a difference of two degrees $d$ -effective zero cycle on the generic fiber and it is an infinite order element on the generic fiber then such elements are parametrized by the complement of a countable union of Zariski closed subsets in the scheme $C^1_d(E_{\ZZ})$. Call it $$Z=\cup_i Z_i\;.$$

This proof is already known in the literature \cite{BH}, but we still include it for the convenience of the reader.

\begin{theorem} The set
$$\bcZ_d^n:=\{(z,\Spec(\QQ))\in C^1_{d,d}(E_{\ZZ}/U)|Supp(z)\subset {E_{\QQ}}, n[z]=0\in \CH^1({E_{\QQ}})\}$$ is a countable union of Zariski closed subsets in the Chow variety $C^1_{d,d}(E_{\ZZ}/U)$ parametrizing the pairs of degree $d$ subvarieties of the arithmetic variety $E_{\ZZ}$.
\end{theorem}
\begin{proof}
This theorem for a surface fibered over $\PR^2$  is proved in the paper by Banerjee and Hoque in \cite{BH}, but we reconstruct the proof for the curve case here for convenience and completeness.
 There are some crucial points to be noted here.
\\
(I) The notion of Hilbert scheme and the Hom scheme makes sense for an arithmetic variety. This is as explained in \cite{FGA}[Chapter: Hilbert schemes and Quot schemes,  5] FGA.
\\
(II) The family of Weil divisors of a smooth fibration over $\Spec(\ZZ)$ is parameterized by a Chow variety, which is actually given by the Picard scheme parameterizing relative Cartier divisors of the same family \cite[Corollary 11.8]{Ry}. In our case the family is $E_{\ZZ}$ which is of finite presentation over $\ZZ$ and is a standard smooth algebra\footnote{in the sense, stack exchange \cite{St}, Definitions 10.136.6 and 29.32.1.} over $\ZZ$. This enables us to formulate the definition of rational equivalence for arithmetic varieties as in \cite[\S 3.3]{GS} in the following way:

Two Weil divisors $D_1,D_2$ are rationally equivalent on a fiber $E_{\QQ}$, if there exists a morphism $$f: \PR^1_{U}\to C^1_{d,d}(E_{\ZZ}/\PR^1_{U})$$ such that
$(f\circ 0)|_{\QQ}=D_1+B$ and $ (f\circ \infty)|_{\QQ}=D_2+B$,
where $B$ is a positive Weil divisor and $0,\infty$ are two fixed sections from $U$ to  $\PR^1_{U}$.

Let us assume that the divisor $D=D^+-D^-$ is rationally equivalent to zero. This means that there exists a map $$f:\PR^1\to C^1_{d,d}({E_{\QQ}})$$ such that
$$f(0)=D^{+}+\gamma\text{ and }f(\infty)=D^{-}+\gamma,$$
where $\gamma$ is a positive divisor in ${E_{\QQ}}$.
In other words, we have the following map:
$$\ev:Hom^v(\PR^1_{U},C^1_{d}(E_{\ZZ}/U))\to C^1_{d}(E_{\ZZ}/U)\times C^1_{d}(E_{\ZZ}/U) $$
 given by $f\mapsto (f(0),f(\infty))$ and that the generic fiber of  $f$ at $\QQ$ is contained in $C^1_{d,d}(E_{\QQ})$.

Let us denote $C^1_{d}(E_{\ZZ}/U)$ by $C^1_d(E_{\ZZ})$ for simplicity.
We now consider the subscheme $U_{v,d}(E_{\ZZ})$ of $ U\times\Hom^v(\PR^1_{U},C^1_{d}(E_{\ZZ}))$ consisting of the pairs $(b,f)$ such that image of $f$ is contained in $C^1_{d}(E_{b})$ (such a universal family exists, for example see \cite[Theorem 1.4]{Ko} or \cite{FGA}, [Chapter on Hilbert schemes and Quot schemes]. Consider its fiber over $\QQ$. This gives a morphism from $U_{v,d}(E_{\ZZ})_{\QQ}$ the generic fiber to
$$ C^1_{d,d}(E_{\QQ})$$
 defined by $$f\mapsto (f(0),f(\infty)).$$
Again, we consider the closed subscheme $\bcV_{d,d}$ of $U\times C^1_{d,d}(E_{\ZZ})$ given by $(b,z_1,z_2)$, where $(z_1,z_2)\in C^1_{d,d}(E_b)$. Suppose that the map from $\bcV_{d,u,d,u}$ to $\bcV_{d+u,u,d+u,u}$ is given by
$$(b,A,C,B,D)\mapsto (b,A+C,C,B+D,D).$$
Let us denote the fiber product by $\bcV$ of $U_{v,d}(E_{\ZZ})$ and $\bcV_{d,u,d,u}$ on $\bcV_{d+u,u,d+u,u}$. Consider the fiber $\bcV_{\QQ}$. If we consider the projection from $\bcV_{\QQ}$ to $ C^1_{d,d}(E_{\QQ})$, then we observe that $A$ and $B$ are supported and rationally equivalent in $E_{\QQ}$. Conversely, if $A$ and $B$ are supported as well as rationally equivalent on $E_{\QQ}$, then we get the map $$f:\PR^1_{\QQ}\to C^1_{d+u,u,d+u,u}(E_{\QQ})$$ of some degree $v$ satisfying
$$f(0)=(A+C,C)\text{ and } f(\infty)=(B+D,D),$$
where $C$ and $D$ are supported on $E_{\QQ}$. This implies that the image of the projection from $\bcV_{\QQ}$ to $C^1_{d,d}(E_{\QQ})$ is a quasi-projective subscheme $W_{d}^{u,v}$ consisting of the tuples $(b,A,B)$ such that $A$ and $B$ are supported on $E_{\QQ}$, and that there exists a map $$f:\PR^1_{\QQ}\to C^1_{d+u,u}(E_{\QQ})$$ such that
$$f(0)=(A+C,C)$$
 and
 $$f(\infty)=(B+D,D)\;.$$
  Here $f$ is of degree $v$, and $C,D$ are supported on $E_b$ and they are of co-dimension $1$ and degree $u$ cycles. This shows that $W_d$ is the union $\cup_{u,v} W_d^{u,v}$. We now prove that the Zariski closure of $W_d^{u,v}$ is in $W_d$ for each $u$ and $v$. For this, we prove the following:
$$W_d^{u,v}=pr_{1,2}(\wt{s}^{-1}(W^{0,v}_{d+u}\times W^{0,v}_u)),$$
where
$$\wt{s}:  C^1_{d,d,u,u}(E_{\QQ})\to  C^1_{d+u,d+u,u,u}(E_{\QQ})$$
defined by
$$\wt{s}(A,B,C,D)=(A+C,B+D,C,D).$$

We assume $(A,B,C,D)\in  C^1_{d,d,u,u}(E_{\QQ})$ in such a way that $\wt{s}(A,B,C,D)\in W^{0,v}_{d+u}\times W^{0,v}_u$. This implies that there exists an element
$$(b,g)\in \PR^1_{\QQ}\times\Hom^v(\PR^1_{\QQ},C^p_{d+u}(E_{\QQ}))$$ and an element
$$(b,h)\in \Hom^v(\PR^1_{\QQ},C^p_{u}(E_{\QQ}))$$ satisfying $$g(0)=A+C,~g(\infty)=B+D \text{ and } h(0)=C,h(\infty)=D$$ as well as the image of $g$ and $h$ are contained in $C^1_{d+u}(E_b)$ and  $C^1_u(E_b)$ respectively.

Also, if $f=g\times h$ then $f\in \Hom^v(\PR^1_{\QQ},C^1_{d+u,u}(E_{\QQ}))$ is such that the image of $f$ is contained in $C^1_{d+u,u}(E_{\QQ})$ and also satisfies the following:
$$f(0)=(A+C,C)\text{ and }(f(\infty))=(B+D,D).$$
This shows that $(A,B)\in W^d_{u,v}$.

On the other hand, if we assume that $(A,B)\in W^d_{u,v}$, then there exists $f\in \Hom^v(\PR^1_{\QQ},C^1_{d+u,u}(E_{\QQ}))$ such that
$$f(0)=(A+C,C)\text{ and }f(\infty)=(B+D,D),$$
and image of $f$ is contained in the Chow scheme of $\bar{E_{\QQ}}$.

We now compose $f$ with the projections to $C^1_{d+u}(E_{\QQ})$ and to $C^1_{u}(E_{\QQ})$ to get a map $g\in \Hom^v(\PR^1_{\QQ},C^1_{d+u}(E_{\QQ}))$ and a map $h\in\Hom^v(\PR^1_{\QQ},C^1_{u}(E_{\QQ}))$ satisfying
$$g(0)=A+C,\quad g(\infty)=B+D$$
and
$$h(0)=C,\quad h(\infty)=D.$$
Also, the image of $g$ and $h$ are contained in the respective Chow varities of the fiber $E_{\QQ}$. Therefore, we have
$$W_d=pr_{1,2}(\wt{s}^{-1}(W_{d+u}\times W_u)).$$

We are now in a position to prove that the closure of $W_d^{0,v} $ is contained in $W_d$. Let $(A,B)$ be a closed point in the closure of ${W_d^{0,v}}$. Let $W$ be an irreducible component of ${W_d^{0,v}}$ whose closure contains $(A,B)$. We assume that $U'$ is an affine neighborhood of $(A,B)$ such that $U'\cap W$ is nonempty. Then there is an irreducible curve $C'$ in $U'$ passing through $(A,B)$. Let $\bar{C'}$ be the Zariski closure of $C'$ in $\overline{W}$. The map
$$e:U_{v,d}(E_{\ZZ})\subset {U}\times \Hom^v(\PR^1_{U},C^1_{d}(E_{\ZZ}))\to C^1_{d,d}(E_{\ZZ})$$
given by
$$(b,f)\mapsto (b,f(0),f(\infty))$$
is regular and if we take the generic fiber of this map, $W_d^{0,v}$ is its image. We now choose a curve $T$ in $U_{v,d}(E_{\QQ})$ such that the closure of $e(T)$ is $\bar C'$.  Let $\wt{T}$  denote the normalization of the Zariski closure of $T$, and $\wt{T_0}$ be the preimage of $T$ in this normalization. Then the regular morphism $\wt{T_0}\to T\to \bar C'$ extends to a regular morphism, when the scalar extends to the field of algebraic numbers. Let this morphism be $\wt{T}_{\QQ}$ to $\bar C'_{\QQ}$. If $(f_{\QQ})$ is a preimage of $(A_{\QQ},B_{\QQ})$, then
$$f_{\QQ}(0)=A_{\QQ}, \quad f_{\QQ}(\infty)=B_{\QQ}$$ and the image of $f_{\QQ}$ is contained in $C^p_{d}(E_{\QQ})$. 
Spreading out $f_{\QQ}$, we have an $f$ such that $$f(0)=A, \quad f(\infty)=B\;.$$
 There is a one-to-one correspondence between $\Spec(\bar\ZZ)$ points of arithmetic varieties and $\bar Q$ points of the corresponding variety over $\bar Q$. Therefore, $A$ and $B$ are rationally equivalent. This completes the proof.
\end{proof}

\textit{Proof of main theorem:}

The condition that the elements are in the rank part of $E(\QQ)$ indicates (by the previous theorem) that they are parametrized by the complement $\cup_{d,n}Z_d^n=Z$. This complement is a countable intersection of Zariski open subsets of the relative Chow variety. Now, we know that there is an infinite order element of the group $E(\QQ)$, so the countable intersection is non-empty. Then if we spread it out over $E_{\ZZ}$ we get a non-trivial element in the Chow group of co-dimension one cycles $CH^1(E_{\ZZ})$. Suppose that for all number fields of the form $\QQ(\sqrt{p^3+ap+b})$ it is zero in the class group of the given number field for $p$ in a Zariski open set $U$ in $\Spec(\ZZ)$. In that case the element say $D$ is such that 
$$D_p=div(f)$$
for a rational function $f$ on the ring of integers of the number field. Since it happens for all such number fields we have a collection of rational functions on some open sets $U_i\subset E_{\ZZ}$, such that $(U_i,f_i)$ forms a Cartier divisor and we have a Line bundle corresponding to $D$ which is trivial on the union $\cup_i U_i$, so that $D$ is rationally equivalent to zero on $CH^1(E_U)$.

Let us give some details on this. We have $f$ is a rational function on the ring of integers, we have $D_p=div(f)$ which gives that $D$ is rationally equivalent to zero on some Zariski open $U$ in $\bcO_K$. This is because $$cl(K)\supset\varprojlim CH^1(E_U)$$ where $E_U$ is the pull-back of $E_{\ZZ}$ over $U\subset Spec(\bcO_K)$. The cycle $D_p$ is in $\varprojlim CH^1(E_U)$. Now $\pi:\Spec(\bcO_K)\to \Spec(\ZZ)$ is a finite morphism. Hence, if we take a open set $V\subset Spec(\ZZ)$ and pull-it back to $U$, we have $D$ restricted to such an open set is rationally equivalent to zero and hence the push-down to $CH^1(E_V)$ is rationally equivalent to zero as well. Hence, we have, by the push-forward pullback formula, 
$$\pi_*\pi^*(D)=nD=0$$
in $Pic^0(E)(\QQ)\cong E(\QQ)$.
Here we assume that the elliptic curve has a $\QQ$ point. Since $D$ is of infinite order, this gives a contradiction as $nD=0$ and hence we have $D_p$ in $Cl(K)$ is non-trivial for infinitely many primes $p$. So, there exists infinitely many $p$ such that $\QQ(\sqrt{p^3+ap+b})$ has class number larger than one.

\subsection*{Acknowledgements}The author thanks SRM AP for hosting the project. The author thanks the anonymous referee for suggestions to improve the paper.

\end{document}